\documentclass[12pt]{amsart}

\usepackage{amssymb}
\usepackage{enumerate}

\makeatletter
\@namedef{subjclassname@2010}{%
  \textup{2010} Mathematics Subject Classification}
\makeatother

\theoremstyle{plain}
\newtheorem{theorem}{Theorem}[section]
\newtheorem{lemma}[theorem]{Lemma}
\newtheorem{corollary}[theorem]{Corollary}
\newtheorem{question}[theorem]{Question}

\newcommand{\CC}{{\mathbb C}}
\newcommand{\DD}{{\mathbb D}}

\newcommand{\TT}{{\mathbb T}}
\newcommand{\cB}{{\mathcal B}}
\newcommand{\cD}{{\mathcal D}}
\newcommand{\cH}{{\mathcal H}}
\DeclareMathOperator{\hol}{\mathrm Hol}
\DeclareMathOperator{\bmoa}{\mathrm BMOA}
\DeclareMathOperator{\vmoa}{\mathrm VMOA}

\begin{document}


\baselineskip=17pt



\title{Linear maps preserving inner functions}

\author[J. Mashreghi]{Javad Mashreghi}
\address{D\'epartement de math\'ematiques et de statistique, Universit\'e Laval,
Qu\'ebec City (Qu\'ebec),  Canada G1V 0A6.}
\email{javad.mashreghi@mat.ulaval.ca}

\author[T. Ransford]{Thomas Ransford}
\address{D\'epartement de math\'ematiques et de statistique, Universit\'e Laval,
Qu\'ebec City (Qu\'ebec),  Canada G1V 0A6.}
\email{thomas.ransford@mat.ulaval.ca}

\date{}

\begin{abstract}
We show that, for many  holomorphic function spaces on the unit disk,
a continuous endomorphism  that sends inner functions to inner functions 
is necessarily a weighted composition operator.
\end{abstract}

\subjclass[2010]{Primary 15A86; Secondary  30J05, 47B33}

\keywords{Inner function, Linear preserver, Composition operator}

\maketitle


\section{Introduction and statement of results}\label{S:intro}

Throughout this article, $\DD$ denotes the open unit disk, $\TT$ is the unit circle, and  $\hol(\DD)$
denotes the space of holomorphic functions on $\DD$, endowed with the
usual topology of uniform convergence on compact sets.

In an earlier work \cite{MR15}, 
the authors used  a version of the Gleason--Kahane--\.Zelazko theorem
to deduce a result about linear maps  that preserve outer functions.
(We shall recall the definition of outer functions in \S\ref{S:io}.)
More precisely, it was shown in \cite[Theorem~3.2]{MR15} that,
if $X$ is a Banach space of holomorphic functions on $\DD$
satisfying certain natural requirements, and if $T:X\to\hol(\DD)$
is a continuous linear map that sends outer functions to outer functions,
then $T$ is necessarily a weighted composition operator. In other words, 
there exist holomorphic functions $\phi:\DD\to\DD$ and $\psi:\DD\to\CC$ such that
\[
Tf=\psi .(f\circ\phi) \quad(f\in X).
\]

At a recent presentation of this work,
a member of the audience raised the question as to what happens
if one replaces `outer' with `inner' throughout. 
(A brief account of inner functions will also
be given in \S\ref{S:io}.) 
The purpose of the present note is to answer this question.
Curiously, the answer is the same: 
a linear map that preserves inner functions has to be a weighted composition operator,
although this time $\phi$ and $\psi$ have to be inner functions. 
Also, the method of proof is completely different.

We now state our results in more precisely.
We consider a Banach space $X$ of holomorphic functions on $\DD$
 satisfying the following conditions:
\begin{enumerate}[(X1)]
\setlength{\itemindent}{12pt}
\item The inclusion map $X\to\hol(\DD)$ is continuous.
\item The polynomials are contained in $X$ and are dense in $X$.
\item $\limsup_{n\to\infty}\|z^n\|_X^{1/n}\le1$.
\end{enumerate}

Examples of spaces $X$ satisfying (X1)--(X3) include:
\begin{itemize}
\item the Hardy spaces $H^p~(1\le p<\infty)$;
\item the Bergman spaces $A^p~(1\le p<\infty)$;
\item the holomorphic Besov spaces $B_p~(1\le p<\infty)$;
\item the weighted Dirichlet spaces $\cD_\alpha~(0\le\alpha\le1)$;
\item the harmonically weighted Dirichlet spaces $\cD(\mu)$;
\item the holomorphic Sobolev spaces $S^p:=\{f:f'\in H^p\}~ (1\le p<\infty)$;
\item the disk algebra $A(\DD)$;
\item the little Bloch space $\cB_0$;
\item the  space $\vmoa$ of functions of vanishing mean oscillation;
\item the de Branges--Rovnyak spaces $\cH(b)$, for non-extreme points $b$  in the unit ball of $H^\infty$.
\end{itemize}
For background on these various spaces, we refer to the books \cite{EKMR14,FM16,Ga07,Zh07}.

The following theorem is our main result.
The conditions (X1)--(X3) easily imply that $X$ contains all rational functions with poles
outside $\overline{\DD}$, and in particular all finite Blaschke products, 
so the statement of the theorem makes sense.

\begin{theorem}\label{T:main}
Let $X\subset\hol(\DD)$ be a Banach space satisfying (X1)--(X3).
Let $T:X\to\hol(\DD)$ be a continuous linear map that maps finite Blaschke products to inner functions. 
\begin{enumerate}[(i)]
\item If $\dim(T(X))>1$, then
 there exist inner functions $\phi,\psi$, with $\phi$ non-constant, such that
\begin{equation}\label{E:main}
Tf=\psi.(f\circ \phi) \quad(f\in X).
\end{equation}
\item If $\dim(T(X))=1$, then there exist an inner function $\psi$ and a point $\alpha\in\TT$
such that
\begin{equation}\label{E:dim1}
Tf=f(\alpha)\psi \quad(f\in X\cap \hol(\overline{\DD})).
\end{equation}
\end{enumerate}
\end{theorem}

Here $\hol(\overline{\DD})$ denotes the set of functions holomorphic on a neighborhood of $\overline{\DD}$.
It is included in \eqref{E:dim1} to ensure that $f(\alpha)$ makes sense. Note that, since $T$ is continuous, \eqref{E:dim1} implies that there exists a constant $C$ such that $|f(\alpha)|\le C\|f\|_X$ for all $f\in X\cap\hol(\overline{\DD})$. Many spaces $X$ do not have this property for any $\alpha\in\TT$, and so for these spaces the case~(ii) never occurs. In particular, this is the situation when $X=H^p,A^p,\cD_\alpha,\cB_0$ and $\vmoa$.
On the other hand, there are spaces for which case~(ii) does occur, notably $A(\DD)$.

Note also that we do not need to assume that $T$ sends general inner functions to inner functions.
In fact it is a consequence of the theorem. This forms part of the next result.

\begin{corollary}\label{C:inner}
Let $X$ and $T$ be as in Theorem~\ref{T:main}(i).
If $f\in X$ is inner, then so is $Tf$. Conversely,
 if  $Tf$ is inner, then so is $f$.
\end{corollary}

Under mild extra assumptions on $T$, we can say more about the functions $\phi$ and $\psi$.
Let $N^+$ denote the Smirnov class, namely the set of all functions $f$ expressible as $f=hg$,
where $h$ is inner and $g$ is outer (together with the function $f\equiv0$).
We remark that $N^+$ contains most of the spaces listed earlier,
namely: $H^p,B_p,\cD_\alpha,\cD(\mu),S^p,A(\DD),\vmoa$ and $\cH(b)$, but not $A^p$ or $\cB_0$.
Further details about $N^+$ are given in \S\ref{S:io}.

\begin{theorem}\label{T:surj}
Let $X,T,\phi$ and $\psi$ be as in Theorem~\ref{T:main}(i).
\begin{enumerate}[(i)]
\item If  $T(X\cap N^+)$ contains an outer function, then $\psi$ is a  unimodular constant.
\item If, further, $T(X)$ separates point of $\DD$, then $\phi$ is an automorphism of $\DD$.
\end{enumerate}
\end{theorem}

We next apply these results to some concrete function spaces.

For $1\le p<\infty$, the \emph{Hardy space} $H^p$ is the set of $f\in\hol(\DD)$
such that
\[
\|f\|_p:=\sup_{r<1}\Bigl(\frac{1}{2\pi}\int_\TT|f(re^{it})|^p\,dt\Bigr)^{1/p}<\infty.
\]

\begin{theorem}\label{T:Hp}
Let  $1\le p<\infty$. The following are equivalent:
\begin{enumerate}[(i)]
\item $T:H^p\to H^p$  is a continuous linear map sending inner functions to inner functions.
\item There exist inner functions $\psi,\phi$, with $\phi$ non-constant, such that 
\[
Tf=\psi.(f\circ\phi) \quad(f\in H^p).
\]
\end{enumerate}
If, further, $T(H^p)=H^p$, then $\psi$ is a unimodular constant and $\phi$ is an automorphism of $\DD$,
and consequently $T$ is invertible.
\end{theorem}

The \emph{Dirichlet space} $\cD$ consists of those $f\in\hol(\DD)$ such that
\[
\cD(f):=\frac{1}{\pi}\int_\DD|f'(z)|^2\,dA(z)<\infty.
\]
It is a subspace of $H^2$, and becomes a Hilbert space when endowed with the norm $\|\cdot\|_\cD$ defined by $\|f\|_{\cD}^2:=\|f\|_{H^2}^2+\cD(f)$.
The only inner functions in $\cD$ are finite Blaschke products 
(see e.g.\ \cite[Corollary~7.6.10]{EKMR14}), so it is natural to consider linear self-maps
of $\cD$ that preserve the set of finite Blaschke products.

\begin{theorem}\label{T:D}
The following are equivalent:
\begin{enumerate}[(i)]
\item $T:\cD\to\cD$ is a continuous linear map sending finite Blaschke products to finite Blaschke products.
\item There are finite Blaschke products $\phi,\psi$, with $\phi$ non-constant, such that
\[
Tf=\psi.(f\circ\phi) \quad(f\in \cD).
\]
\end{enumerate}
If, further, $T(\cD)=\cD$, then $\psi$ is a unimodular constant and $\phi$ is an automorphism of $\DD$,
and consequently $T$ is invertible.
\end{theorem}

We conclude by mentioning that there is an extensive literature about  composition operators
generated by inner functions. 
These tend to arise when considering composition operators that are also isometries
(see e.g.\ \cite{MV06}).
It is interesting that the same operators also arise in the solution to a problem about linear preservers.


\section{Background on inner and outer functions}\label{S:io}

In this section we summarize some basic facts about inner and outer functions.
Unless otherwise indicated, the proofs may be found in Duren's book \cite{Du70}.

\subsection{Inner functions}

Let $f\in\hol(\DD)$ and let $e^{it}\in\TT$. 
We write
\[
f^*(e^{it}):=\lim_{r\to1^-}f(re^{it})
\]
if this limit exists. By a well-known theorem of Fatou, if $f$ is bounded and
holomorphic on $\DD$, then $f^*$ exists a.e.\ on $\TT$.

A function $h\in\hol(\DD)$ is called  \emph{inner} if it is bounded and satisfies
and $|h^*(e^{it})|=1$ a.e.\ on $\TT$. 
In this case, either $h$ is a unimodular constant or $|h(z)|<1$ for all $z\in\DD$.

A basic example of an inner function is a  \emph{Blaschke product}, namely
\[
B(z):=c\prod_{n\ge1} \frac{|a_n|}{a_n}\frac{a_n-z}{1-\overline{a}_nz}
\quad(z\in\DD),
\]
where $(a_n)$ is a (finite or infinite) sequence in $\DD$ satisfying $\sum_n(1-|a_n|)<\infty$,
and $c$ is a unimodular constant. A further example is a so-called \emph{singular inner function}, 
defined by
\[
S(z):=\exp\Bigl(-\int_\TT \frac{e^{it}+z}{e^{it}-z}\,d\sigma(t)\Bigr)
\quad(z\in\DD),
\]
where $\sigma$ is a finite positive measure on $\TT$ 
that is singular with respect to Lebesgue measure. 

If $B,S$ are as above, then clearly their product is also an inner function. 
Conversely, every inner function
$h$ can be decomposed as $h=BS$, 
where $B$ is a Blaschke product and $S$ is a singular function. 
Moreover, this decomposition is unique: the data $(a_n),c,\sigma$ are uniquely determined by $h$.
This is part of the celebrated canonical factorization theorem.

We shall need the following result about division of inner functions.

\begin{lemma}\label{L:quotient}
Let $h_0,h_1$ be inner functions. Suppose that $h_0(h_1/h_0)^k$ is an inner function
for each integer $k\ge0$. Then $h_1/h_0$ is an inner function.
\end{lemma}

\begin{proof}
Set $h_k:=h_0(h_1/h_0)^k$. By hypothesis $h_k$ is inner, so we may factorize it as $h_k=B_kS_k$,
a Blaschke product times a singular factor. 
Let $\mu_k$ be the positive measure on $\overline{\DD}$ formed 
by taking a point mass at each zero of $B_k$ (counted according to multiplicity) 
plus the singular measure defining $S_k$. 
The relation $h_k=h_0(h_1/h_0)^k$ then implies that $\mu_0+k(\mu_1-\mu_0)=\mu_k\ge0$. 
Dividing $k$ and then letting $k\to\infty$, 
we deduce that $\mu_1-\mu_0\ge0$, which amounts to saying that $h_1/h_0$ is an inner function.
\end{proof}

As already remarked, the product of two inner functions is inner.
It is also true that the composition of two (non-constant) inner functions is inner,
though this is not quite so obvious. It is part of the next result, 
whose proof is based on the theorem of Lindel\"of on asymptotic values.

\begin{theorem}\label{T:Lindelof}
Let $f$ be a bounded holomorphic function on $\DD$ and let $h$ be a non-constant inner function.
Then $(f\circ h)^*=f^*\circ h^*$ a.e.\ on $\TT$. Consequently $(f\circ h)$ is inner iff $f$ is inner.
\end{theorem}

\begin{proof}
See e.g.\ \cite[Proposition 2.25]{CM95}.
\end{proof}

For $a\in\DD$, let us write $\phi_a(z):=(a-z)/(1-\overline{a}z)$. 
The following result is a weak form of a theorem of Frostman. 

\begin{theorem}\label{T:Frostman}
Let $h$ be an inner function. Then, for all $a$ in a dense subset of $\DD$,
the composition $\phi_a\circ h$ is a Blaschke product.
\end{theorem}

\begin{proof}
See e.g.\ \cite[Chapter II, Theorem~6.4]{Ga07}.
\end{proof}


\subsection{Outer functions and the Smirnov class}

A function $g\in\hol(\DD)$  is called  \emph{outer} if there exists
$G:\TT\to[0,\infty)$ with $\log G\in L^1(\TT)$ such that
\[
g(z)=\exp\Bigl(\int_0^{2\pi}\frac{e^{it}+z}{e^{it}-z}
\log G(e^{it})\,\frac{dt}{2\pi}\Bigr)\
\quad(z\in\DD).
\]
Its radial limits $g^*(e^{it})$ then exist a.e.\ on $\TT$ and satisfy $|g^*|=G$ a.e.

If a function $f$ can be written as a product $f=hg$, 
with $h$ inner and $g$ outer, then this factorization is unique. 
The set of functions $f$ that can be factored in this way is called the \emph{Smirnov class}, 
denoted $N^+$. 
Usually the function $f\equiv0$ is also included in $N^+$. 
Then $N^+$ becomes an algebra. 

There are a number of other characterizations of $N^+$. 
For example, $f\in N^+$ iff it is the quotient 
of a bounded holomorphic function by a bounded outer function. 
Also, $f\in N^+$ if and only if $f\in\hol(\DD)$ and
the family $\{\log^+|f_r|:0<r<1\}$ is uniformly integrable on $\TT$.
(Here, as usual, $f_r(e^{it}):=f(re^{it})$).

Finally, in this section, we note that the Smirnov class  behaves well with respect to composition.

\begin{theorem}\label{T:N+comp}
If $f\in N^+$ and $h:\DD\to\DD$ is holomorphic, then $f\circ h\in N^+$.
\end{theorem}

\begin{proof}
See e.g.\ \cite[\S2.6]{CKS00}.
\end{proof}


\section{Proofs}\label{S:proofs}

\begin{proof}[Proof of Theorem~\ref{T:main}]
For each $k\ge0$, the function $z^k$ is a finite Blaschke product. Consequently $h_k:=T(z^k)$ is inner.
Also, for each $a\in\DD$, the function $\phi_a(z):=(a-z)/(1-\overline{a}z)$ is a finite Blaschke product,
so $T(\phi_a)$ is inner. Now, for each $a\in\DD$,
\[
\phi_a(z)=\sum_{k=0}^\infty (a\overline{a}^kz^k-\overline{a}^kz^{k+1}) \quad(z\in\DD).
\]
By the condition (X3), the series converges in $X$, and by (X1) it converges to $\phi_a$.  Consequently,
\[
(T\phi_a)(z)=\sum_{k=0}^\infty (a\overline{a}^kh_k(z)-\overline{a}^kh_{k+1}(z)) \quad(z\in\DD).
\]

Fix countable dense subsets 
$D_1$ of $(0,1)$ and $D_2$ of $\TT$. Then there exists a subset $E$ of $\TT$ of measure zero
such that, for all $a=r\zeta~(r\in D_1,~\zeta\in D_2)$,
the radial limit  $(T\phi_a)^*$  exists at each point of $\TT\setminus E$ and has modulus~$1$. Adding a further set of measure zero to $E$, if necessary, we may further suppose that, for all $k\ge0$, the radial limits  $h_k^*$ exist at each point of $\TT\setminus E$ and have modulus~$1$. Hence, on $\TT\setminus E$, we have, for all $r\in D_1$ and $\zeta\in D_2$,
\[
\Bigl|\sum_{k=0}^\infty (r^{k+1}\zeta^{1-k} h_k^*-r^k\zeta^{-k}h_{k+1}^*)\Bigr|^2=1.
\]
Expanding, we get
\[
\begin{split}
\sum_{k,l=0}^\infty \Bigl(&r^{k+l+2}\zeta^{l-k} h_k^*\overline{h_l^*}
+r^{k+l}\zeta^{l-k}h_{k+1}^*\overline{h_{l+1}^*}\\
&-r^{k+l+1}\zeta^{l-k+1}h_k^*\overline{h_{l+1}^*}
-r^{k+l+1}\zeta^{l-k-1}h_{k+1}^*\overline{h_l^*}\Bigr)=1.
\end{split}
\]
As this holds for all $r\in D_1$ and $\zeta\in D_2$,
the coefficient of $r^0\zeta^0$ must be $1$, and the coefficients of $r^n\zeta^m$ for all other pairs $n,m$ must be zero. This leads  to the following relations (holding on $\TT\setminus E$):
\[
h_{k-1}^*\overline{h_{l-1}^*}+h_{k+1}^*\overline{h_{l+1}^*}=2h_k^*\overline{h_l^*} \quad(k,l\ge1).
\]
Note that the two summands on the left-hand side both have modulus $1$, and they add up to the right-hand side which has modulus $2$. It follows that the summands on the left-hand side are in fact equal, and so we deduce that
\[
h_k^*\overline{h_l}^*=h_{k-1}^*\overline{h_{l-1}^*} \quad(k,l\ge1).
\]
In particular, taking $l=k-1$, we see that $h_k^*\overline{h_{k-1}^*}=h_1^*\overline{h_0^*}$ for all $k\ge1$.
We deduce that $h_k^*=h_0^*(h_1^*/h_0^*)^k$ on $\TT\setminus E$  for all $k\ge0$.
Inner functions whose radial boundary values agree a.e.\ on $\TT$ are equal everywhere on $\DD$.
Consequently $h_k=h_0(h_1/h_0)^k$  on $\DD$ for all $k\ge0$.

By Lemma~\ref{L:quotient}, the quotient  $h_1/h_0$ is an inner function.
Set $\psi:=h_0$ and $\phi:=h_1/h_0$. 
By what we have proved, $\phi$ and $\psi$ are inner functions, and $T(z^k)=\psi\phi^k$ for all $k\ge0$. 
Therefore $T(f)=\psi.(f\circ\phi)$ for each polynomial~$f$.

There are now two cases to consider. 
If $\phi$ is non-constant, then,
since polynomials are dense in $X$ 
and $T:X\to\hol(\DD)$ is continuous, 
we have $Tf=\psi.(f\circ\phi)$ for all $f\in X$. 
If, on the other hand, $\phi$ is a unimodular constant $\alpha$,
then $Tf=f(\alpha)\psi$ for all polynomials $f$,
and hence also for all $f\in X\cap \hol(\overline{\DD})$, because
the Taylor series of each such $f$ converges to $f$ in  the norm of $X$.
Clearly we have $\dim(T(X))>1$ in the first case and $\dim(T(X))=1$ in the second.
This yields the cases (i) and (ii) as stated in the theorem.
\end{proof}

\begin{proof}[Proof of Corollary~\ref{C:inner}]
By Theorem~\ref{T:Lindelof}, if $f$ is inner, then $f\circ \phi$ is inner, and hence $Tf=\psi.(f\circ\phi)$ is inner as well. 

Conversely, if $Tf$ is inner, then clearly $Tf/\psi=(f\circ \phi)$ is inner. If we knew that $f$ was bounded, then we could apply Theorem~\ref{T:Lindelof} to deduce that $f$ is inner. However, all we know, \textit{a priori}, is that $f$ is bounded on $\phi(\DD)$. To show that $f$ is bounded on $\DD$, we need to prove that $\phi(\DD)$ is dense in $\DD$. For this we use Theorem~\ref{T:Frostman}. According to that theorem, for a dense set of $a\in\DD$, the composition $\phi_a\circ\phi$ is a Blaschke product. As $\phi_a\circ\phi$ non-constant, it must have a zero, which amounts to saying that $a\in\phi(\DD)$. Thus $\phi(\DD)$ is indeed dense in $\DD$, and the result is proved.
\end{proof}

\begin{proof}[Proof of Theorem~\ref{T:surj}]
(i) By Theorem~\ref{T:N+comp}, if $f\in N^+$ then $f\circ\phi\in N^+$. Consequently,
\[
T(X\cap N^+)\subset \psi. N^+.
\]
Since inner-outer factorization is unique, the only way that the right-hand side can contain an outer function is if $\psi$ is a unimodular constant.

(ii) Assume now that $\psi$ is constant. If $T(X)$ separates points of $\DD$, then clearly $\phi$ has to be injective.
Using Theorem~\ref{T:Frostman}, we know that there exists an automorphism $\phi_a$ of $\DD$ such that $\phi_a\circ \phi$ is a Blaschke product. Then $\phi_a\circ\phi$ is also injective, so it must be a Blaschke product of degree one,
in words, an automorphism of $\DD$. Therefore $\phi$ itself is an automorphism of $\DD$.
\end{proof}

\begin{proof}[Proof of Theorem~\ref{T:Hp}]
The fact that (i) implies (ii) follows immediately from Theorem~\ref{T:main}.
(Note that, for $X=H^p$, case (ii) in Theorem~\ref{T:main} never occurs, as remarked immediately after 
the statement of that theorem.)

The fact that (ii) implies (i) is a combination of two results.
The first is the well-known fact that $Tf:=\psi.(f\circ \phi)$ is a bounded operator on $H^p$
(this is essentially Littlewood's subordination theorem, see e.g.\ \cite[Theorem~1.7]{Du70}).
The second is the fact that $T$ preserves inner functions, which is an easy  consequence of Theorem~\ref{T:Lindelof}.

The last part of the theorem follows from Theorem~\ref{T:surj},
together with the fact that $H^p\subset N^+$ for all $p\ge1$.
\end{proof}

\begin{proof}[Proof of Theorem~\ref{T:D}]
To show that (i) implies (ii), we first use Theorem~\ref{T:main} to show that 
there exist inner functions $\phi,\psi$, with $\phi$ non-constant, such
that $Tf=\psi.(f\circ\phi)$ for all $f\in\cD$. (Note that, just as for Hardy spaces,
case (ii) of Theorem~\ref{T:main} never occurs when $X=\cD$.)
Also, we have  $\psi=T(1)$ and $\psi.\phi=T(z)$,
so, as $T$ maps finite Blaschke products to finite Blaschke products,
 both $\phi$ and $\psi$ are finite Blaschke products.

For the reverse implication, that (ii) implies (i), the only point in doubt is 
whether $Tf:=\psi.(f\circ\phi)$ maps $\cD$ boundedly into itself.
That this is indeed true follows from the fact that, if $\phi$ is a Blaschke product
of degree~$n$, then $\cD(f\circ \phi)=n\cD(f)$ for all $f\in\cD$ (see e.g.\ \cite[Lemma~6.2.2]{EKMR14}).

Once again, the last part of the theorem is a consequence of Theorem~\ref{T:surj},
together with the fact that $\cD\subset N^+$.
\end{proof}


\section{Concluding remarks and questions}\label{S:conclusion}

\subsection{Non-separable spaces}
The proof of Theorem~\ref{T:main} does not work for $H^\infty$, nor for $\cB$ (the Bloch space) or $\bmoa$
(functions of  bounded mean oscillation), because the polynomials are not dense in any of these spaces.

\begin{question}
Is Theorem~\ref{T:main} still valid for $X=H^\infty,\cB$ and $\bmoa$?
\end{question}

\subsection{Automatic continuity}
It was shown in \cite{MR15} that every linear map $T:H^p\to H^p$ preserving outer functions, whether it is continuous or not,  is necessarily a weighted composition operator. A similar result for the Dirichlet space (with outer functions replaced by nowhere-vanishing functions) was recently obtained in \cite{MRR17}. These results are quite a bit more difficult
without the continuity of $T$. We are led to pose the following question.

\begin{question}
Do Theorems~\ref{T:Hp} and \ref{T:D} hold without the continuity of $T$?
\end{question}

\subsection{Singular inner functions}
As remarked in Corollary~\ref{C:inner}, the fact that finite Blaschke products are mapped to inner functions implies that the same is true for general inner functions. What if we start from a different subclass? For example:

\begin{question}
Let $T:H^2\to H^2$ be a continuous linear map such that $Tf$ is inner whenever $f$ is a singular inner function. Does $T$ necessarily map all inner functions to inner functions?
\end{question}

In this context,  is perhaps worth noting that finite Blaschke products are `far' from all the other inner functions.  For example, it is not hard to see that the set of finite Blaschke products is a closed subset of the set of inner functions in the norm of $H^\infty$, and  that it is a distance at least one from the set of all other inner functions.

\subsection*{Acknowledgements}
JM was supported by an NSERC Discovery Grant. 
TR was supported by an NSERC Discovery Grant and a Canada Research Chair.


\end{document}